\documentclass[12pt]{amsart}
\usepackage{amscd}
\usepackage{amssymb}
\usepackage{a4wide}
\usepackage{amstext}
\usepackage{amsthm}
\usepackage{mathrsfs}
\usepackage{color}

\renewcommand{\phi}{\varphi}

\newcommand{\co}{\mathbb{C}}

\newcommand{\Z}{\mathbb{Z}}

\newcommand{\R}{\mathbb{R}}
\newcommand{\rea}{{\rm Re}\,}
\newcommand{\ima}{{\rm Im}\,}
\newcommand{\ff}{\mathcal{F}}

\newtheorem{Thm}{Theorem}[section]
\newtheorem{theorem}[Thm]{Theorem}
\newtheorem{lemma}[Thm]{Lemma}

\newtheorem{mainthm}{Theorem}

\begin{document}
\sloppy

\title[]{Irregular sampling for hyperbolic secant type functions}
\author{Anton Baranov}
\address{Department of Mathematics and Mechanics\\ St. Petersburg State University\\
St. Petersburg, Russia}
\email{anton.d.baranov@gmail.com} 
\author{Yurii Belov}
\address{Department of Mathematics and Computer Science \\ St. Petersburg State University\\
St. Petersburg, Russia}
\email{j\_b\_juri\_belov@mail.ru}



\thanks{The work is supported by Russian Science Foundation project 19-71-30002.}

\begin{abstract}
We study Gabor frames in the case when the window function
is of hyperbolic secant type, i.e., $g(x) = (e^{ax}+e^{-bx})^{-1}$, $\rea a, \rea b>0$.
A criterion for half-irregular sampling is obtained:  for a separated $\Lambda\subset\mathbb{R}$ the Gabor system
$\mathcal{G}(g, \Lambda \times \alpha\Z)$ is a frame in $L^2(\R)$ if and only if 
$D^-(\Lambda) >\alpha$ where $D^-(\Lambda)$ is the usual (Beurling) lower density of $\Lambda$.
This extends a result by Gr\"ochenig, Romero, and St\"ockler which applies to the case of a standard hyperbolic secant. 
Also, a full description of complete interpolating sequences for the shift-invariant space generated by $g$
is given.
\end{abstract}

\maketitle


\section{Introduction and main results}
\label{section1}

Given a function $g\in  L^2(\R)$, its {\it time-frequency shifts} are defined as
$$
M_yT_xg(t) = e^{-2\pi i yt}g(x-t), \qquad x,y\in\R.
$$
Expansions of a function (signal) with respect to the time-frequency shifts of a given {\it window function} $g$
with a certain ``good localization''
form the basis for the time-frequency analysis. 

 For a discrete subset of the time-frequency plane $S\subset\mathbb{R}^2$, consider {\it the Gabor system}
$$
\mathcal{G}(g, S) = \big\{M_yT_xg: \ (x,y) \in S \big\}.
$$
Then a basic question is to determine whether $\mathcal{G}(g, S)$ is a
frame in $L^2(\R)$, that is, whether there exist  $A,B>0$ such that
$$
A\|f\|^2 \le \sum_{(x,y) \in S} |(f, M_yT_xg)|^2 \le B \|f\|^2, \qquad f\in L^2(\R).
$$
Of special interest is the lattice case $S = \alpha\Z\times\beta\Z$, $\alpha, \beta>0$.
We refer to the monograph \cite{gr01} for a comprehensive exposition of the theory.
It is well known that the density condition $\alpha\beta\leq1$ is necessary for a Gabor system to be a frame \cite{gr01}.

Even in the lattice case the description of the set of $(\alpha, \beta)$ such that 
$\mathcal{G}(g, \alpha\Z\times\beta\Z)$ is a frame is known only for very special classes of functions.
The archetypical example here is the Gaussian $g(x) = e^{-\pi x^2}$
for which the problem of Gabor frames can be reduced to sampling in the Bargmann--Fock space $\mathcal{F}$
via the Bargmann transform. Then one can apply the description of sampling sets in  $\mathcal{F}$ 
 due to Lyubarskii \cite{lyu}, Seip \cite{seip}, and Seip and Wallst\'en \cite{sw} to show 
 that $\mathcal{G}(g, \alpha\Z \times \beta\Z)$ is a frame if and only if
$\alpha\beta<1$. The same is true for symmetric exponential functions
$e^{-|x|}$.  For the trunctated exponential $\chi_{(0, \infty)}(x) e^{-|x|}$,
the system $\mathcal{G}(g, \alpha\Z \times \beta\Z)$ is a frame if and only if $\alpha\beta\leq1$.

In 2002 Jannsen \cite{jan1} proved that the hyperbolic  secant $(e^{ax}+e^{-ax})^{-1}$, $a>0$,
generates a frame if and only if $\alpha\beta<1$.

Only recently analogous results were established for classes of functions in place of individual functions.
In 2011  Gr\"ochenig and St\"ockler \cite{gs1} proved that 
$\mathcal{G}(g, \alpha\Z \times \beta\Z)$ is a frame if and only if
$\alpha\beta<1$ when $g$ is a totally positive function of finite type, while in 2018
Gr\"ochenig, Romero, St\"ockler \cite{grs1} showed that the same is true for Gaussian totally positive 
functions of finite type.  In \cite{gr23} Gr\"ochenig extended these results to general totally positive functions 
under the condition that $\alpha\beta\in\mathbb{Q}$. 
Recently, in \cite{uz} shift-invariant spaces and Gabor frames were studied for the case when $g$ is a rational function
of $e^{ax}$, $a>0$.

Another class of functions for which the Gabor frames of the form
$\mathcal{G}(g, \alpha\Z \times \beta\Z)$  can be decribed was introduced by
Belov, Kulikov and Lyubarskii \cite{bkl}. They proved that a Herglotz function of finite type,
 i.e., $\sum_{k=1}^N\frac{a_k}{x-iw_k}$, where $a_k,\,w_k>0$,
generates a frame if and only if $\alpha\beta\leq1$. Finally, in a recent preprint \cite{belsem} 
the problem is solved for the case of a shifted sinc-kernel. 
In general, the set of parameters $(\alpha,\beta)$ such that 
$\mathcal{G}(g, \alpha\Z \times \beta\Z)$ is a frame can be very complicated, 
e.g., this happens for $g(x)=\chi_{[0,1]}(x)$ \cite{dau} and for the 
Haar function $g(x)=\chi_{[0,1\slash2]}(x)-\chi_{[1\slash2,1]}(x)$ \cite{dai}.

The so called case of irregular sampling where the sets are no longer assumed to be lattices is in general even more 
complicated. The only case when this problem is solved completely is the case of the Gaussian,
and the proof is again based on the description of sampling sets in the Bargmann--Fock space
\cite{lyu, seip, sw}.

\begin{mainthm}
\label{thm:main0}
Let $S \subset \R^2$ be a separated set and let $g(x) = e^{-\pi x^2}$. Then $\mathcal{G}(g, S)$
is a frame for $L^2(\R)$ if and only if 
$$
D_2^-(S) = 
\liminf_{r\to\infty} \frac{\inf_{z\in\co} {\rm card}\,(S \cap B_r(z))}{\pi r^2} >1.
$$
\end{mainthm}

Here $B_r(z)$ denotes the disc in $\mathbb{C}$ with the center $z$ and radius $r$ 
and $D_2^-(S) $ denotes the lower planar density of a set $S$.

For rectangular sets $\Lambda\times M$, $\Lambda,M\subset\mathbb{R}$,
and a Cauchy kernel, the description of Gabor frames was found in \cite{bkl2}.

For the case of ``half-regular'' sampling (i.e., $S=\Lambda\times\alpha\mathbb{Z}$) the problem was solved 
for totaly positive functions of finite (Gaussian) type in \cite{grs1,gs1}  and for the secant in
\cite[Theorem 6.2]{grs2}.
Recall that a function $g$ is a totally positive function of finite  type or a Gaussian totally positive function of finite  type
if, respectively,  
$\hat g(t) = \prod_{j=1}^n (1+i \delta_j t)^{-1}$ or 
$\hat g(t) = e^{-ct^2} \prod_{j=1}^n (1+i \delta_j t)^{-1}$, $c>0$, $\delta_j\in \R$.
It should be mentioned that in \cite[Theorem 6.2]{grs2} the problem of half-regular sampling is solved
for a more general class of multi-window Gabor frames. 

\begin{mainthm}
\label{thm:mainA}
Assume that g is a totally positive function of finite  type, or a Gaussian totally positive function of finite  type,
or the hyperbolic secant $(e^{ax}+e^{-ax})^{-1}$.  Let
$\Lambda\subset \R$ be a separated set. Then $\mathcal{G}(g, \Lambda \times \Z)$
is a frame for $L^2(\R)$ if and only if $D^-(\Lambda) > 1$.
\end{mainthm} 

Here $D^+({\Lambda})$ and $D^{-}({\Lambda})$ denote the  usual
upper and lower (linear) Beurling densities on $\R$,
$$
\begin{aligned}
D^+(\Lambda) & = \lim_{r\to\infty} \frac{\sup_{x\in\R} {\rm card}\,(\Lambda\cap [x, x+r])}{r}, \\
D^-(\Lambda) & = \lim_{r\to\infty} \frac{\inf_{x\in\R} {\rm card}\,(\Lambda\cap [x, x+r])}{r}.
\end{aligned}
$$

In the present paper we consider a more general class of the hyperbolic secant type 
functions. For $a, b \in \mathbb{C}$ 
with $\rea a>0$ and  $\rea b>0$, put
\begin{equation}
\label{hyp}
g(x) = \frac{1}{e^{ax}+e^{-bx}}. 
\end{equation} 

The main result of the paper is the following analog of Theorem B for such functions.
Note that the functions of the form \eqref{hyp} are not (Gaussian) totally positive functions 
when $a, b\notin \R$; for $a,b>0$ the function $g$ is totally positive, but not totally positive of finite type.  
Thus, the main novelty of the paper is in the case of nonreal parameters. 

\begin{theorem}
\label{main1}
 Let  $\Lambda\subset \R$ be a separated set and let $g$ be given by \eqref{hyp}. Then the following are equivalent:
 
 1. $\mathcal{G}(g, \Lambda \times \alpha\Z)$ is a frame in $L^2(\R)$\textup;
 
 2. $\mathcal{G}(g, \alpha\Z \times \Lambda)$ is a frame in $L^2(\R)$\textup;
 
 3. $D^-(\Lambda) >\alpha$.
  \end{theorem}
  
  In Theorem \ref{main1} we restrict ourselves with the case of separated sequences, since any frame of the form
  $\mathcal{G}(g, \Lambda \times \alpha\Z)$  contains a frame $\mathcal{G}(g, \Lambda' \times \alpha\Z)$ where
  $\Lambda'\subset \Lambda$  is separated.
  
  Note also that all results (with the same proofs) remain true for a slightly wider class of functions  
  $g(x) = (\alpha e^{ax}+\beta e^{-bx})^{-1}$, where $\rea a, \, \rea b>0$, $\alpha, \beta \in \mathbb{C}\setminus\{0\}$,
  subject to the condition that $\alpha e^{ax}+\beta e^{-bx}\ne 0$ on $\R$.
  
  Equivalence of Statements 1 and 2  follows from the fact that the systems $\mathcal{G}(g, \Lambda \times M)$ 
  and $\mathcal{G}(\hat g, M\times \Lambda)$ are or are not frames simultaneously and that $\hat g$ also is a
  hyperbolic secant type function.
  
One of the main tools for the study of Gabor frames is the so called shift-invariant space generated by the window 
function $g$. In what follows we will assume that $g$ belongs to the so-called {\it Wiener amalgam space} $W_0$, that is, 
$g$ is continuous on $\R$ and $\sum_{k\in \Z} \max_{x\in [k, k+1]} |g(x)| <\infty$. With any $g\in W_0$
satisfying the condition
\begin{equation}
\label{stab}
C_1\|(c_n)\|_{\ell^2} \le \bigg\|\sum_{n\in \Z} c_ng(x-n)\bigg\|_{L^2(\R)}  \le C_2\|(c_n)\|_{\ell^2}, \qquad (c_n) \in\ell^2(\Z),
\end{equation}
for some constants $C_1, C_2>0$, we associate the {\it shift-invariant space}
$$
V^2(g) =\big\{f(x) = \sum_{n\in \Z} c_ng(x-n) : \ (c_n) \in\ell^2(\Z) \big\} \subset L^2(\R).
$$

Recall that a sequence $\Lambda = \{\lambda_n\}_{n\in\Z} \subset \R$ is said to be {\it sampling for $V^2(g)$}
if there exist $A,B>0$ such that
$$
   A\|f\|^2 \le \sum_{n\in \Z} |f(\lambda_n)|^2 \le B\|f\|^2, \qquad f\in V^2(g).
$$
The following theorem \cite[Theorem 2.3]{grs1} establishes a relation between the description of Gabor
frames and the sampling in $V^2(g)$: 
 
\begin{mainthm}
\label{thm:mainB}
Let $g\in W_0$ satisfy \eqref{stab}. Then the system  $\mathcal{G}(g, \Lambda\times \mathbb{Z})$ is a frame in 
$L^2(\mathbb{R})$ if and only if $\Lambda + x$ is a sampling set for $V^2(g)$ for all $x\in \R$.
\end{mainthm} 

Thus, the condition of Theorem \ref{thm:mainB}  says that, for any $x\in \R$, there exist $A_x,B_x>0$ such that
\begin{equation}
\label{frx}
A_x\|f\|^2 \le \sum_{n\in \Z} |f(\lambda_n+x)|^2 \le B_x \|f\|^2, \qquad f\in V^2(g).
\end{equation}
However, it is shown in \cite[Theorem 2.3]{grs1} that once \eqref{frx} holds with constants depending on $x$, 
a similar inequality holds with uniform bounds on sampling constants, i.e., there exist $A,B>0$ such that
$$
A\|f\|^2 \le \sum_{n\in \Z} |f(\lambda_n+x)|^2 \le B \|f\|^2, \qquad f\in V^2(g), \ x\in\R.
$$

We say that $\Lambda$ is {\it interpolating for $V^2(g)$} if for every $(a_n) \in \ell^2(\Z)$ 
there exists $f\in V^2(g)$ such that $f(\lambda_n) = a_n$. If $\Lambda$ is both sampling and interpolating
(or, equivalently, the solution of the interpolation problem is unique), we say that
$\Lambda$ is a {\it complete interpolating sequence}.  

Sampling and interpolation in the spaces $V^2(g)$ 
for totally positive functions and Gaussian totally positive functions of finite type
were studied in \cite{gs1, grs1}, where it was shown, in particular, that every separated sequence $\Lambda \subset \R$ 
 with $D^{-}({\Lambda})> 1$  is sampling for $V^2(g)$. In \cite{bbg} a full description of complete interpolating sequences
 was found for $V^2(g)$ in the Gaussian case $g(x) = e^{-cx^2}$. We find a similar description for
 hyperbolic secant type functions:
 
 \begin{theorem}
 \label{main2}
 Let $g$ be given by \eqref{hyp} with  $\rea a, \, \rea b>0$. Then an increasing sequence $\Lambda
\subset\R $ is a complete interpolating sequence for $V^2(g)$ if and only if $\Lambda$ 
is separated and there exists an enumeration $\Lambda=\{\lambda_n\}_{n\in \Z}$, $\lambda_n = n+\delta_n$, $n\in\Z$,
such that
   \begin{enumerate}
   \item[(a)] $(\delta_n)_{n\in\Z}\in \ell^\infty$\textup;
   \item[(b)] there exists $N\geq 1$  such that
$$
\sup_{n\in\Z} \frac{1}{N} \Big|\sum_{k=n+1}^{n+N} \Big(\delta_k -
\frac{\rea a-\rea b}{2(\rea a +\rea b)} \Big) \Big|  <\frac{1}{2}.
$$ 
 \end{enumerate}
 
Also, if $\Lambda$ is a separated sequence with $D^{-}({\Lambda})> 1$, then $\Lambda$ contains 
a complete interpolating sequence for $V^2(g)$, while any separated sequence with
 $D^{+}({\Lambda})< 1$ can be appended to a complete interpolating sequence for $V^2(g)$.
 \end{theorem}
 
 As in \cite{bbg}, the proof is based on a reduction to the study of sampling (or interpolating) sequences
 in a certain Fock type space. More precisely, we first establish 
 a natural isomorphism between $V^2(g)$ and a certain space of discrete Cauchy transforms and then
 show that this space coincides with a Fock type space of functions with singularities at $0$ and $\infty$.
 
 We conclude with the following conjecture that for the secant type function one has a complete description 
 of irregular sampling similar to that for the Gaussian:
 \medskip
 \\
 {\bf Conjecture.} {\it    Let $g$ be given by \eqref{hyp} and let  $S$ be a separated subset of $\mathbb{R}^2$.  
 Then $\mathcal{G}(g, S)$ is a frame for $L^2(\R)$ if and only if  $D^-(S)>1$.  }
 \medskip
 \\
 \textbf{Notations.} In what follows we write $U(x)\lesssim V(x)$ if 
there is a constant $C$ such that $U(x)\leq CV(x)$ holds for all $x$ 
in the set in question. We write $U(x)\asymp V(x)$ if both $U(x)\lesssim V(x)$ and
$V(x)\lesssim U(x)$. The standard Landau notations
$O$ and $o$ also will be used.
\medskip
\\
\textbf{Acknowledgement.} The authors are grateful to Karlheinz Gr\"ochenig for numerous 
helpful comments. 
\bigskip
\bigskip
 

\section{Relation between $V^2(g)$ and a space of Cauchy transforms} 
\label{cauchy}
 
Throughout the rest of the paper we may assume without loss of generality that $a=1$  in \eqref{hyp}
(otherwise consider the window function $g(x+c)$ for an appropriate choice of $c\in\R$). 
 
Let $f\in V^2(g)$. Then there exists $(c_n) \in \ell^2(\Z)$ such that
$$
 f(x) =  \sum_{n\in \Z} c_ng(x-n) =  \sum_{n\in \Z} \frac{c_n}{e^{a(x-n)} +  e^{-b(x-n)}}   
 =   e^{bx}\sum_{n\in \Z} \frac{c_n e^{an} }{e^{(a+b)x} + e^{(a+b)n}}.
$$  
Introducing a new variable $z = e^{(a+b)x} $ we will consider this sum as a discrete Cauchy transform. 
More precisely, consider the zero order entire function
$$
A(z) = \prod_{n=1}^\infty \bigg(1 +\frac{z}{e^{(a+b)n}}\bigg)
$$
and put 
\begin{equation}
\label{gen}
G(z) = (1+z) A(z)A(1/z).
\end{equation}
Note that $G$ has two essential singularities at $0$ and $\infty$. Finally, let
$$
\mathcal{H}_{a,b} = \bigg\{ F(z) = G(z) \sum_{n\in \Z} \frac{c_n e^{an} }{z+ e^{(a+b)n}}: \ (c_n)\in\ell^2(\Z) \bigg\}
$$
and define $\|f\|_{\mathcal{H}_{a,b}} = \|(c_n)\|_{\ell^2(\Z)} $. 
It is easy to see that $\mathcal{H}_{a,b}$ is a Hilbert space 
with respect to this norm, which consists of functions holomorphic in $\mathbb{C} \setminus \{0\}$. The mapping
\begin{equation}
\label{daf}
f(x) = \sum_{n\in \Z} c_n g(x-n) \in V^2(g) \ \mapsto \ F(z) = G(z) \sum_{n\in \Z} \frac{c_n e^{an}}{z+e^{(a+b)n}}
\end{equation}
is an isomorphism of $V^2(g)$ onto $\mathcal{H}_{a,b}$. 

The space $\mathcal{H}_{a,b}$ is a reproducing kernel Hilbert space. It is immediate that the reproducing kernel 
of  $\mathcal{H}_{a,b}$ at a point $w$ is given by
$$
k^{a,b}_w(z) =   \sum_{n\in\Z} \frac{G(z)\overline{G(w)}e^{2n \rea a}}{(\bar w+ e^{(\bar a+\bar b)n})(z+ e^{(a+b)n})};
$$
in particular, for  $w=-w_m = - e^{(a+b)m}$ one has
$$
k^{a,b}_{-w_m} (z) =\overline{G'(-w_m)}e^{2 m\rea a} \frac{G(z)}{z+w_m}.
$$
Note that the reproducing kernels $\{k^{a,b}_{-w_m}\}_{m\in\Z}$ form an orthogonal basis of $\mathcal{H}_{a,b}$.

General Hilbert spaces of discrete Cauchy transforms were introduced in \cite{bms}. In the case when the poles accumulate
only to infinity corresponding spaces of entire functions were studied in detail in \cite{abb, loc2, bar18, bar23}.
\bigskip


\section{Small Fock type spaces} 
In this Section we introduce small Fock type spaces $\mathcal{F}_{\beta, \gamma}$. In the next Section we show that these spaces coincide with spaces $\mathcal{H}_{a,b}$
with equivalence of norms.
For $\beta>0$, $\gamma \in\R$, put $\phi(z) = \beta \log^2 |z| + \gamma \log|z|$ and let
$$
\mathcal{F}_{\beta, \gamma} = 
\bigg\{F\in Hol(\mathbb{C} \setminus\{0\})  : \  \|F\|_{\beta, \gamma}^2=  \int_{\mathbb{C}} 
|F(z)|^2 e^{-2\phi(z)} dm_2(z) <\infty \bigg\},
$$
where $m_2$ is the Lebesgue measure $\frac{1}{\pi} dxdy$.
Then $\mathcal{F}_{\beta, \gamma}$ is a reproducing kernel Hilbert space. 

Clearly, the system $\{z^n\}_{n\in\mathbb{Z}}$ is an orthogonal basis in $\mathcal{F}_{\beta, \gamma}$ and it is easy to see that, if 
$F(z) = \sum_{n\in\Z} c_nz^n$, then
\begin{equation}
\label{pow}
\|F\|_{\mathcal{F}_{\beta, \gamma}}^2 =
\sqrt{\frac{2\pi}{\beta}} \sum_{n\in \Z} \exp\Big(\frac{(n+1-\gamma)^2}{2\beta}\Big) |c_n|^2.
\end{equation}
\medskip


\subsection{Estimate of the norm of a reproducing kernel}

\begin{lemma}
\label{norm}
Let $K_w^{\beta, \gamma}$ be the reproducing kernel of $\mathcal{F}_{\beta, \gamma}$ at a point $w$
and $\phi(z) = \beta \log^2 |z| + \gamma \log|z|$. Then
\begin{equation}
\label{norm1}
\|K_w^{\beta, \gamma}\|_{\beta, \gamma} \asymp \frac{e^{\phi(w)}}{|w|}, \qquad  w\in \mathbb{C} \setminus\{0\}.
\end{equation}
\end{lemma}

\begin{proof}
The proof of this lemma can be found in \cite{bl} in a similar setting. We include here a simple proof to make the paper
more self-contained.  

We start with an estimate from above. Let $w\neq 0$ and set $D_w = \{z\in \mathbb{C}: |z-w| <|w|/10\}$. Put
$\Phi(z) = \exp(\beta \log^2 z + \gamma \log z)$, where $\log z$ is an analytic  branch of the logarithm analytic in $D_w$
such that $|\arg z| < 2\pi$. Then $|\Phi(z)| \asymp e^{\phi(z)}$, $z\in D_w$. Let $F\in \mathcal{F}_{\beta, \gamma}$,
$\|F\|_{\beta, \gamma} \le 1$. Then, by subharmonicity, 
$$
1\gtrsim \int_{D_w}\bigg| \frac{f(z)}{\Phi(z)}\bigg|^2 dm_2(z) \gtrsim |w|^2 \bigg| \frac{f(w)}{\Phi(w)}\bigg|^2.
$$
We conclude that $|f(w)| \lesssim |\Phi(w)|/|w|$ whenever $\|F\|_{\beta, \gamma} \le 1$, whence
$\|K_w^{\beta, \gamma}\|_{\beta, \gamma} \lesssim e^{\phi(w)}/|w|$.

For the estimate from below take $F(z) = z^n$, $n\in\Z$, and note that $\|F\|^2_{\beta, \gamma} 
\asymp \exp\big(\frac{(n+1-\gamma)^2}{2\beta}\big)$. Now let $w\in \mathbb{C} \setminus\{0\}$, $r=|w|$
and take $n = [2\beta \log r] = 2\beta\log r -\delta$, $0\le \delta<1$.
Then
$$
\begin{aligned}
\|K_w^{\beta, \gamma}\|_{\beta, \gamma} & \ge |F(w)|/\|F\|_{\beta, \gamma}  \gtrsim
\exp\Big( n\log r - \frac{(n+1-\gamma)^2}{4\beta} \Big) \\
& = 
\exp\Big( 2\beta \log^2 r - \delta\log r -\beta \Big(\frac{ 2\beta\log r -\delta +1-\gamma}{2\beta}   \Big)^2    \Big)  \\
& = 
\exp\big(\beta \log^2 r +(\gamma-1)\log r +O(1) \big) \asymp  \frac{e^{\phi(w)}}{|w|}.
\end{aligned}
$$
\end{proof}


\subsection{Complete interpolating sequences in $\mathcal{F}_{\beta, 0}$}

Most of the radial Fock type spaces of entire functions 
do not have complete interpolating sequences (or, equivalently, Riesz bases 
of normalized reproducing kernels). There exists a vast literature on this subject (see, e.g., \cite{bbb, bdkel, bl, is1, is2}
and references therein). 
However, in small Fock spaces with the weights $e^{-2\phi}$, where $\phi(r) = O(\log^2 r)$ and
$\phi$ has some additional regularity
(in particular,  in Fock spaces of entire functions with $\phi(r) = \beta (\log^+ r)^2$)
there exist complete interpolating sequences. 
This was noted for the first time by Borichev and Lyubarskii in \cite{bl}.  Later
a full description of complete interpolating sequences for the Fock space with 
$\phi(r) = \beta (\log^+ r)^2$ was obtained in \cite{bdhk}.

We will need a similar description for the case of ``two-sided'' spaces 
of functions with singularities at $0$ and $\infty$. Let us introduce the notions of sampling and interpolation for the
space $\mathcal{F}_{\beta, \gamma}$. A sequence $\{\lambda_n\}$ is said to be {\it sampling} for 
$\mathcal{F}_{\beta, \gamma}$ if 
$$
\sum_{n} (1+|\lambda_n|^2) e^{-2\phi(\lambda_n)} |F(\lambda_n)|^2 \asymp \|F\|^2_{\beta, \gamma}, 
\qquad F\in \mathcal{F}_{\beta, \gamma}. 
$$
Analogously, $\{\lambda_n\}$ is {\it interpolating} for $\ff_{\beta, \gamma}$ if for every sequence $\{a_n\}$ satisfying
$$
\sum_{n} (1+|\lambda_n|^2) e^{-2\phi(\lambda_n)} |a_n|^2 <\infty
$$ 
there exists
$F\in\ff_{\beta, \gamma}$ with $F(\lambda_n) = a_n$. In view of the estimate 
\eqref{norm1} for the norm of the reproducing kernel,
we immediately see that $\{\lambda_n\}$ is sampling for $\ff_{\beta, \gamma}$ if and only if 
$\{\tilde K^{\beta,\gamma}_{\lambda_n}\}$  is
a frame in  $\mathcal{F}_{\beta, \gamma}$, where $\tilde K^{\beta,\gamma}_{\lambda} = 
K^{\beta,\gamma}_{\lambda}/\| K^{\beta, \gamma}_{\lambda}\|_{\beta, \gamma}$. It is easy to show that
$\{\lambda_n\}$ is interpolating if and only if $\{\tilde K^{\beta,\gamma}_{\lambda_n}\}$ is a Riesz sequence.
Finally, $\{\lambda_n\}$ is said to be a {\it complete interpolating sequence} for  $\mathcal{F}_{\beta, \gamma}$
if it is both interpolating and sampling (equivalently,  $\{\tilde K^{\beta,\gamma}_{\lambda_n}\}$ is a Riesz basis 
in $\mathcal{F}_{\beta, \gamma}$). 

The criterion for being a complete interpolating sequence in $\mathcal{F}_{\beta, 0}$
will be given in terms of mean deviations from a distinguished complete interpolating sequence
$\{e^{\frac{n}{2\beta}}\}_{n\in \Z}$ which was found already in \cite{bl} for the case of the corresponding space
of entire functions.

\begin{theorem} 
\label{compin}  
A sequence $\Lambda$ is a complete interpolating sequence for  $\mathcal{F}_{\beta, 0}$
if and only if  there exists an enumeration $\Lambda = \{\lambda_n\}_{n\in \Z}$ such that
$0<|\lambda_n|\le |\lambda_{n+1}|$,
$$
\lambda_n = e^{\frac{n}{2\beta}} e^{\delta_n} e^{  i \theta_n}
$$ 
with $\delta_n, \theta_n \in\R$, and the following three conditions are satisfied\textup:
   \begin{enumerate}
   \item[(i)] $\{\lambda_n\}$ is logarithmically separated, i.e., there exists
     $c >0$ such that $|\lambda_m-\lambda_n| \ge c\max(|\lambda_m|, |\lambda_n|)$, $m\ne n$\textup;
   \item[(ii)] $(\delta_n)_{n\in \Z} \in \ell^\infty$\textup;
   \item[(iii)] there exists $N\geq 1$ and $\delta>0$ such that
\begin{equation}
\label{avd}
\sup_{n\in \Z} \frac{1}{N}\Big|\sum_{k=n+1}^{n+N}\delta_k\Big| <\frac{1}{4\beta}.
\end{equation}
   \end{enumerate}
   \end{theorem}

This theorem was proved in \cite{bdhk} for Fock type spaces of entire functions (one-sided power series)
with $\phi(r) = \beta (\log^+r)^2$,
but the proof applies without substantial changes to the case of Laurent series. Also, the ``two-sided'' 
case (at least for $\lambda_n\in\R$) is contained in \cite{bbg}, where the relation between 
complete   interpolating sequences for $\mathcal{F}_{\beta, 0}$ and for Gaussian invariant spaces was established. 
Now we discuss this relation in details. 
\medskip


\subsection{Gaussian invariant spaces and Fock spaces}
\label{gaus}

In this subsection we recall the isomorphism between Gaussian invariant spaces and Fock spaces
of the form $\mathcal{F}_{\beta, 0}$. This 
isomorphism (essentially different from the one in Section \ref{cauchy}) was the key step 
in the study of complete interpolating sequences for the
Gaussian invariant spaces in \cite{bbg}. Here we consider 
a slightly more general situation of complex Gaussian functions. 
For $\alpha>0$ and $\sigma \in \R$ let $g_\sigma (x) = e^{-(\alpha+i\sigma) x^2}$. 
Then any $f\in V^2(g_\sigma)$ is of the form
$$
f(x) = \sum_{n\in \Z} c_ne^{- (\alpha+i\sigma)(x-n)^2} = e^{-(\alpha+i\sigma) x^2}\sum_{n\in \Z} c_n e^{-\alpha n^2 -i\sigma n^2}
\, e^{2(\alpha+i\sigma)nx},  \qquad (c_n) \in \ell^2(\Z),
$$
and $\|f\|_{V^2(g_\sigma)} \asymp \|(c_n)\|_{\ell^2(\Z)}$.
Introducing a new variable $z=e^{2(\alpha+i\sigma)x}$, we see that $ V^2(g_\sigma)$ 
is isomorphic to the space of functions representable as
$$
\bigg\{ F(z) = \sum_{n\in \Z} d_n z^n: \ \  \|F\|^2 : = \sum_{n\in \Z}
|d_n|^2 e^{2\alpha (n+1)^2} <\infty\bigg\},
$$
which is, in view of \eqref{pow}, isomorphic to $\mathcal{F}_{\beta, 0}$ with $\beta = \frac{1}{4\alpha}$.

In view of this isomorphism  a sequence 
$\{x_m\} \subset\R$ is sampling for $V(g_\sigma)$ if and only if the sequence 
$\{e^{(\alpha+i\sigma) x_m}\}$ is sampling for $\mathcal{F}_{\beta, 0}$.

We will need the following theorem analogous to Theorems B and C. 
It is by no means new, but we include a very short proof to make the exposition self-contained. 

\begin{mainthm}
\label{thm:mainD}
Let $\Lambda\subset \R$ be a separated set. Then the following statements are equivalent\textup:

1. $\mathcal{G}(g_\sigma, \Lambda\times \mathbb{Z})$ is a frame in $L^2(\mathbb{R})$\textup;

2. $\Lambda + x$ is a sampling set for $V^2(g_\sigma)$ for all $x\in \R$ with uniform bounds on sampling constants\textup;

3. $D^-(\Lambda) > 1$.
\end{mainthm}  

\begin{proof}
Statements 2 and 3 are equivalent by Theorem \ref{thm:mainB}. Let us show that Statements 1 and 3 are equivalent.
Note that for $f\in L^2(\R)$ we have
$$
\Big(e^{-(\alpha+i\sigma) (x-n)^2} \, e^{-2\pi i \lambda x}, f \Big) = e^{-i\sigma n^2}
\Big(e^{-\alpha (x-n)^2}\, e^{-2\pi i (\lambda -\frac{\sigma n}{\pi})x}, e^{i\sigma x^2} f \Big).
$$
Since the map $f\mapsto e^{i\sigma x^2} f $ is a unitary operator on $L^2(\R)$, we conclude that
$\mathcal{G}(g_\sigma, \Lambda\times \mathbb{Z})$ is a frame in $L^2(\mathbb{R})$ if and only 
$\mathcal{G}(g_0, S)$ is a frame, where 
$$
S = \Big\{ \Big(n, \lambda - \frac{\sigma n}{\pi}\Big):\, n\in \Z, \, \lambda\in \Lambda \Big\}
$$
and $g_0 (x) = e^{-\alpha x^2}$. It is easy to see that $D_2^-(S) = D^-(\Lambda)$.
Indeed, the mapping $(x,y) \mapsto \big(x, y+\frac{\sigma x}{\pi}\big)$
has determinant 1, whence it preserves the area of the parallelograms forming the lattice.

By Theorem \ref{thm:main0}, $\mathcal{G}(g_0, S)$ is a frame if and only if $D_2^-(S) >-1$.
Thus, $\mathcal{G}(g_\sigma, \Lambda\times \mathbb{Z})$ is a frame in $L^2(\mathbb{R})$ if and only if $D^-(\Lambda) >1$.
\end{proof}
\bigskip


\section{The space $\mathcal{H}_{a,b}$ coincide with $\mathcal{F}_{\beta, \gamma}$} 

The following theorem is the most important technical step of the proof of Theorem \ref{main1}:

\begin{theorem}
\label{coin}
We have $\mathcal{H}_{a,b} = \mathcal{F}_{\beta, \gamma}$ as sets with equivalence of norms for
$$ 
\beta = \frac{1}{2 \rea (a+b)}, \qquad \gamma = \frac{1}{2} +\frac{\rea a}{\rea (a+b)}.
$$
\end{theorem}

In the proof we will use several lemmas. Recall that $k_w^{a,b}$ denotes the reproducing kernel of  $\mathcal{H}_{a,b}$.
To simplify the notations, let $s = \rea a$, $t=\rea b$.

\begin{lemma}
\label{lem0}
Let $G$ be given by \eqref{gen} and $W = \{-w_m\}_{m\in\Z}$, $w_m = e^{(a+b)m}$. Then, for $w\in \mathbb{C}\setminus\{0\}$,
\begin{equation}
\label{gen1}
|G(w)| \asymp \exp\Big( \frac{\log^2 |w|}{2(s+t)} -\frac{\log|w|}{2}\Big) {\rm dist}(w, W),
\end{equation}
\begin{equation}
\label{norm2}
\|k_w^{a,b}\| \asymp \exp\Big( \frac{\log^2 |w|}{2(s+t)} +\Big(\frac{s}{s+t} -\frac{1}{2} \Big) \log|w|\Big).
\end{equation}
\end{lemma}

\begin{proof}
Assume first that $\log |w| = (s+t)m+\delta$, where $m\ge 1$ and $|\delta|\le (s+t)/2$. Then 
${\rm dist}(w, W)\asymp |w+w_m|$. We have $|A(1/w)| \asymp 1$  and 
$$
|G(w)|\asymp |w| \cdot \prod_{n=1}^{m-1}\frac{|w|}{|w_n|} \cdot \frac{|w+w_m|}{|w_m|} = |w|^m \cdot
|w+w_m| \cdot \prod_{n=1}^m e^{-(s+t)n}.
$$
since $|w_m| = e^{(s+t)m}$. Hence, 
$$
\begin{aligned}
\log|G(w)| - \log |w+w_m|  & = m\log|w| - \frac{m(m+1)(s+t)}{2} \\
& =  \frac{\log|w| -\delta}{s+t}\log|w| -
\frac{\log|w| -\delta}{2}\Big(\frac{\log|w| -\delta}{s+t} +1\Big)  \\ 
& = \frac{\log^2|w|}{2(s+t)}- \frac{\log|w|}{2} +O(1).
\end{aligned}
$$
If $|w| <1$, we use the fact that $|G(w)| = |wG(1/w)|$ and ${\rm dist}(1/w, W) \asymp {\rm dist}(w, W)/|w|^2$.

Now let $\log |w| = (s+t)m+\delta$, where $m\in \Z$ and $|\delta|\le (s+t)/2$. Assume also that $w\notin W$.
Then 
$$
\|k_w^{a,b}\|  = |G(w)|\Bigg(\sum_{n\in\Z} \frac{e^{2sn}}{|w+w_n|^2}\Bigg)^{1/2} \asymp  \frac{|G(w)|e^{sm}}{|w+w_m|}
\asymp \frac{|G(w)|}{|w+w_m|}e^{\frac{s}{s+t} \log|w|}.
$$
Then the estimate for $\|k_w^{a,b}\|$ follows from the estimate for $|G(w)|$ (the case $w\in W$ follows by continuity).
\end{proof}
  
\begin{lemma}
\label{lem2}
For any $\gamma\in\R$ there exists $c>0$ such that $\{F(cz) : \ F\in  \mathcal{F}_{\beta, \gamma} \} 
= \mathcal{F}_{\beta, 0}$.

In particular, a set $\{u_m\}$ is sampling \textup(interpolating\textup) for 
$\mathcal{F}_{\beta, \gamma} $ if and only if $\{c u_m\}$ is sampling
 \textup(respectively, interpolating\textup)  for $\mathcal{F}_{\beta, 0}$. 
\end{lemma}

\begin{proof} 
We have, by the change of variable $z=cw$,
$$
\|F\|_{\beta, \gamma}^2= c^2
\int_{\mathbb{C}} |F(cw)|^2 e^{-2 \beta (\log|w| +\log c)^2 -2\gamma (\log|w| +\log c)} dm_2(w).
$$
It remains to choose $c$ so that $2\beta \log c+\gamma=0$.

We have $\|k^{\beta, \gamma}_{cw}\|_{\beta, \gamma} \asymp \|k^{\beta, 0}_w\|_{\beta, 0}$ and it follows that
a set $\{u_m\}$ is sampling for $\mathcal{F}_{\beta, \gamma} $ if and only if $\{c u_m\}$ is sampling for
$\mathcal{F}_{\beta, 0}$. 
\end{proof}

\begin{proof}[Proof of Theorem \ref{coin}]
First we show that $\{-w_m\}_{m\in\Z}$, where $w_m = e^{(a+b)m}$, is a complete interpolating sequence
for $\mathcal{F}_{\beta, \gamma} $. 
By Lemma \ref{lem2}, this is equivalent to the property that $\{-cw_m\}_{m\in\Z}$ 
with $2\beta \log c+\gamma=0$ is a complete interpolating 
sequence for  $\mathcal{F}_{\beta, 0} $. Note that
the description of complete interpolating sequences for $\mathcal{F}_{\beta, 0} $ in Theorem \ref{compin} 
essentially depends only on the moduli of the points, and so it is sufficient to show that the sequence
$|-cw_m| = c e^{(s+t)m} = c e^{\frac{m}{2\beta}}$ is a complete interpolating 
sequence for  $\mathcal{F}_{\beta, 0}$. We have
$$
c|w_m| = \exp \Big( \frac{m}{2\beta}-\frac{\gamma}{2\beta}  \Big).
$$
By Theorem \ref{compin}, the set $\{\exp \big( \frac{m}{2\beta} - s  \big)\}_{m\in\Z}$ (a shift of the standard complete
interpolating sequence) is complete interpolating if and only if $s\ne \frac{1}{4\beta} +\frac{k}{2\beta}$ for some $k\in \Z$, since 
in this case one can achieve  \eqref{avd} shifting the enumeration if necessary.
Note that
$$
\frac{\gamma}{2\beta} = \frac{1}{4\beta} + \frac{s}{s+t}\cdot\frac{1}{2\beta} \ne \frac{1}{4\beta} +\frac{k}{2\beta},
$$
since $0< \frac{s}{s+t} <1$. 

We conclude that $\{-w_m\}_{m\in\Z}$ is a complete interpolating sequence
for $\mathcal{F}_{\beta, \gamma} $. Therefore, the system of normalized kernels
\begin{equation}
\label{rb1}
\big\{ K^{\beta, \gamma}_{-w_m}/ \|K^{\beta, \gamma}_{-w_m}\|_{\beta, \gamma} \big\}_{m\in\Z}
\end{equation}
is a Riesz basis for $\mathcal{F}_{\beta, \gamma} $.
Let us consider its biorthogonal system (which is a Riesz basis as well).
It follows from \eqref{gen1} that $\frac{G(z)}{z+w_m} \in \mathcal{F}_{\beta, \gamma}$ for any $m$. Hence, the biorthogonal 
system to \eqref{rb1} is given by $\{G_m\}_{m\in\Z}$, 
$$
G_m(z) = \frac{\|K^{\beta, \gamma}_{-w_m}\|_{\beta, \gamma}}{G'(-w_m)} \cdot \frac{G(z)}{z+w_m}.
$$
In view of \eqref{norm1} and \eqref{gen1} we have
$$
\frac{\|K^{\beta, \gamma}_{-w_m}\|_{\beta, \gamma}}{|G'(-w_m)|} \asymp 
 \exp\Big( \frac{\log^2 |w_m|}{2(s+t)} +\Big( \frac{s}{s+t} - \frac{1}{2}\Big) \log|w_m|\Big) \cdot
\exp\Big( -\frac{\log^2 |w_m|}{2(s+t)} +\frac{\log|w_m|}{2}\Big) =e^{sm},
$$
since $w_m = e^{(s+t)m}$. Since $G_m$ is a Riesz basis in $\mathcal{F}_{\beta, \gamma}$, it follows that 
$\mathcal{F}_{\beta, \gamma} $ coincides with the set of functions representable as
$$
F(z) = G(z) \sum_{m\in\Z} c_m \frac{e^{sm}}{z+w_m},\qquad (c_m)\in\ell^2(\Z),
$$
and $\|F\|_{\beta,\gamma} \asymp \|(c_m)\|_{\ell^2(\Z)}$. 
Recall that $s=\rea a$. Thus, $\mathcal{F}_{\beta, \gamma} = \mathcal{H}_{a,b}$
with equivalence of norms.
\end{proof}
\medskip

Now we can prove Theorem  \ref{main2}. 

\begin{proof}[Proof of Theorem \ref{main2}]
By the isomorphisms between $V^2(g)$, $\mathcal{H}_{a,b}$ and $\mathcal{F}_{\beta, \gamma}$, and by Lemma \ref{lem2}
we have that $\{\lambda_n\}_{n\in\Z}$ is a complete interpolating sequence for $V^2(g)$ if and only if
$\{e^{(a+b)\lambda_n} \}_{n\in\Z}$ is a complete interpolating sequence for  $\mathcal{F}_{\beta, \gamma}$
and if and only if $\{ce^{(a+b)\lambda_n} \}_{n\in\Z}$ is a complete interpolating sequence for  $\mathcal{F}_{\beta, 0}$,
where 
$$
\beta = \frac{1}{2(s+t)}, \qquad \gamma = \frac{1}{2} +\frac{s}{s+t}, \qquad \log c  = -\frac{\gamma}{2\beta}.
$$
Since $\{\lambda_n\}_{n\in\Z}$ is separated and complete interpolating sequences in $\mathcal{F}_{\beta, 0}$ are rotation 
invariant, one more equivalent condition is that 
$\{ce^{(s+t)\lambda_n} \}_{n\in\Z}$ is a complete interpolating sequence for  $\mathcal{F}_{\beta, 0}$,
where  $s = \rea a$, $t=\rea b$.
We write $\lambda_n = n+\delta_n$, $n\in\Z$. Then
$$
ce^{(s+t)\lambda_n} = \exp\Big( \frac{n}{2\beta} +\frac{\delta_n -\gamma}{2\beta}\Big).
$$
By Theorem \ref{compin}, this sequence is a complete interpolating sequence for  $\mathcal{F}_{\beta, 0}$ 
if and only if after an appropriate change
of enumeration (which in our case means simply a shift of the index) one has
$$
\sup_{n\in \Z} \frac{1}{N}\Big|\sum_{k=n+1}^{n+N}\Big(\delta_k -\gamma\Big)\Big|  = 
\sup_{n\in \Z} \frac{1}{N}\Big|\sum_{k=n+1}^{n+N}\Big(\delta_k -\frac{1}{2} - \frac{s}{s+t}\Big) \Big| 
<\frac{1}{2}.
$$
Shifting the index by 1, we obtain another enumeration for which
$$
\sup_{n\in \Z} \frac{1}{N}\Big|\sum_{k=n+1}^{n+N}\Big(\delta_k +\frac{1}{2} - \frac{s}{s+t} \Big) \Big|
<\frac{1}{2}.
$$

The fact that any separated sequence $\Lambda$ with $D^{-}({\Lambda})> 1$ contains 
a complete interpolating sequence for $V^2(g)$, while any separated sequence with
$D^{+}({\Lambda})< 1$ can be appended to a complete interpolating sequence for $V^2(g)$ follows in
exactly the same way as in \cite{bbg}.
\end{proof}

We see that, as in the case of the Gaussian, among all lattices obtained by a shift of $\Z$ there is exactly one 
which fails to be a complete interpolating sequence in $V^2(g)$. 
\bigskip
 

\section{Proof of Theorem \ref{main1}} 

By standard rescaling argument we can assume that $\alpha=1$. As above we write $s = \rea a$, $t=\rea b$.

\begin{lemma}
\label{lem3}
A sequence $\Lambda = \{\lambda_m\}$ is sampling for $V^2(g)$ with $g$ given by \eqref{hyp} 
if and only the sequence 
$\{e^{(a+b) \lambda_m}\}$ is sampling for $\mathcal{H}_{a,b} = \mathcal{F}_{\beta, \gamma}$ 
with $\beta = \frac{1}{2(s+t)}$, $\gamma = \frac{1}{2} +\frac{s}{s+t}$.
Moreover, there exist positive constants $c=c(a,b)$, $C=C(a,b)$ such that if 
$$
   A\|f\|^2 \le \sum_{n\in \Z} |f(\lambda_n)|^2 \le B\|f\|^2, \qquad f\in V^2(g),
$$
then, for $u_m = e^{(a+b) \lambda_m}$,
$$
\tilde A \|F\|_{\mathcal{H}_{a,b}}^2 \le 
\sum_m \|k_{u_m}^2\|_{\mathcal{H}_{a,b}}^{-2} |F(u_m)|^2 \le \tilde B \|F\|_{\mathcal{H}_{a,b}}^2,
\qquad F\in \mathcal{H}_{a,b},
$$
where $\tilde A \ge cA$, $\tilde B \le CB$.
\end{lemma}

\begin{proof} 
Let $u_m =e^{(a+b) \lambda_m}$ and let $f \in V^2(g)$ and $F\in \mathcal{H}_{a,b}$ be related via \eqref{daf}. 
Then 
$$
\|f\| \asymp \|F\|_{\mathcal{H}_{a,b}}, \qquad \sum_{m} |f(\lambda_m)|^2 \asymp \sum_m
|G(u_m)|^{-2}e^{2t\lambda_m} |F(u_m)|^2.
$$

Let us show that ${\rm dist}\, (u_m, W) \asymp |u_m|$, where $W = \{-w_k\}$.
Since for $\rea a, \rea b>0$ we have $e^{ax}+e^{-bx} \ne 0$, $x\in \R$, there exists $\delta>0$ such that 
$|e^{(a+b)x} +1| \ge \delta$,  $x\in \R$. Hence, $|e^{(a+b) \lambda_m}+e^{(a+b) k}| = 
|u_m| \cdot |e^{(a+b) (k-\lambda_m)} +1| \ge\delta |u_m|$.
Since $|u_m| = e^{(s+t)\lambda_m}$, it follows from \eqref{gen1} and \eqref{norm2} that
$$
\begin{aligned}
|G(u_m)|e^{-t\lambda_m} & \asymp
\exp\Big( \frac{\log^2 |u_m|}{2(s+t)} +\frac{\log|u_m|}{2} - t\lambda_m\Big) \\
& =
\exp\Big( \frac{\log^2 |u_m|}{2(s+t)} +\frac{s-t}{2}\lambda_m\Big) \asymp \|k^{a,b}_{u_m}\|_{\mathcal{H}_{a,b}}
\end{aligned}
$$
with the constant depending on $a,b$ only. Thus, $\|F\|_{\mathcal{H}_{a,b}}^2
\asymp \sum_m \|k^{a,b}_{u_m}\|^{-2}_{\mathcal{H}_{a,b}} |F(u_m)|^2$.
\end{proof}

\begin{lemma}
\label{lem4}
The sequence $\{e^{(a+b) (\lambda_m+x)}\}$ is sampling for $ \mathcal{F}_{\beta, \gamma}$ for any $x\in \R$ 
if and only if $\{e^{(a+b) (\lambda_m+x)}\}$ is sampling for $ \mathcal{F}_{\beta, 0}$  for any $x\in \R$.
\end{lemma}

\begin{proof} 
By Lemma \ref{lem2}
$\mathcal{F}_{\beta, 0} = \{F(cz) : \ F\in  \mathcal{F}_{\beta, \gamma} \}$ for some $c>0$, and a set 
$\{u_m\}$ is sampling for $\mathcal{F}_{\beta, \gamma} $ if and only if $\{c u_m\}$ is sampling for
$\mathcal{F}_{\beta, 0}$. 

Let $c=e^{y_0}$ and choose $x_0\in\R$ such that $y_0 = \rea (a+b) x_0$. Since
$\{e^{(a+b) (\lambda_m+x -x_0)}\}$ is sampling for $ \mathcal{F}_{\beta, \gamma}$ for any $x\in \R$ 
we conclude that $\{c e^{(a+b) (\lambda_m+x-x_0)}\} = 
\{e^{-i \ima (a+b)x_0} e^{(a+b) (\lambda_m+x)}\}$ is sampling for $ \mathcal{F}_{\beta, 0}$ for any $x\in \R$ and vice versa.
The unimodular factor is harmless since $ \mathcal{F}_{\beta, 0}$ is rotation invariant. 
\end{proof}
\medskip

\begin{proof}[Proof of Theorem \ref{main1}]
 We will show that the following statements are equivalent. 
 \medskip

 (i) Let $\mathcal{G}(g, \Lambda\times \mathbb{Z})$ be a frame in $L^2(\mathbb{R})$.
 \medskip

 By Theorem C this is equivalent to 
 \medskip
 
 (ii) $\Lambda + x$ is a sampling set for $V^2(g)$ for all $x\in \R$. 
\medskip

By Lemmas \ref{lem3} and \ref{lem4} this is equivalent to the fact that 
\medskip

(iii) The sequence
$\{e^{(a+b) (\lambda_m+x)}\}$ is sampling for $\mathcal{F}_{\beta, 0}$, where $\beta = \frac{1}{2\rea (a+b)}$ for any 
\medskip
$x\in \R$.

Using the isomorphism between the Gaussian shift invariant space $V^2(h)$ with $h(x) = e^{-(a+b) x^2} $
and $ \mathcal{F}_{\beta, 0}$ (see Subsection \ref{gaus}), we conclude that statement (iii) is equivalent to 
\medskip

(iv) $\Lambda + x$ is a sampling set for $V^2(h)$ for all $x\in \R$.
\medskip

Applying Theorem D for the window $h$ we see that (iv) is equivalent to 
\medskip

(v) $D^{-}({\Lambda})> 1$.
\medskip

Theorem  \ref{main1} is proved.
\end{proof}

\end{document}